\documentclass[a4paper]{article}

\usepackage[all]{xy}
\usepackage[leqno]{amsmath}
\usepackage{amssymb}
\usepackage[only,mapsfrom,heavycircles,lightning]{stmaryrd}
\usepackage{latexsym}
\usepackage{titlesec} 
\usepackage{paralist}
\usepackage{tocloft}
\usepackage{xcolor}
\usepackage{lmodern}
\usepackage[T1]{fontenc}
\usepackage[utf8]{inputenc}
\usepackage{calc}
\usepackage{hieroglf} 
\usepackage[normalem]{ulem}

\usepackage[hidelinks=true]{hyperref} 

\newlength{\theorempostskipamount}
\makeatletter

\newenvironment{theorem}[1][]
{\paragraph{Theorem~\theparagraph{\normalfont \spcifnec{#1}}.} \it}
{\vspace{\the\theorempostskipamount}}

\newenvironment{proposition}[1][]
{\paragraph{Proposition~\theparagraph{\normalfont \spcifnec{#1}}.} \it}
{\vspace{\the\theorempostskipamount}}

\newenvironment{remark}[1][]
{\paragraph{Remark~\theparagraph{\normalfont \spcifnec{#1}}.} }
{\vspace{\the\theorempostskipamount}}

\newenvironment{corollary}[1][]
{\paragraph{Corollary~\theparagraph{\normalfont \spcifnec{#1}}.} \it}
{\vspace{\the\theorempostskipamount}}

\newenvironment{definition}[1][]
{\paragraph{Definition} {\normalfont #1} \it}
{\vspace{\the\theorempostskipamount}}
\def\definition{\@ifnextchar[{\@definitionopt}{\@definitionplain}}
\def\@definitionplain{\paragraph{Definition} \it}
\def\@definitionopt[#1]{\paragraph{Definition \normalfont #1}  \it}

\newenvironment{question}[1][]
{\paragraph{Question} {\normalfont #1} \it}
{\vspace{\the\theorempostskipamount}}
\def\question{\@ifnextchar[{\@questionopt}{\@questionplain}}
\def\@questionplain{\paragraph{Question} \it}
\def\@questionopt[#1]{\paragraph{Question \normalfont #1}  \it}

\newenvironment{problem}[1][]
{\paragraph{Problem} {\normalfont #1} \it}
{\vspace{\the\problempostskipamount}}
\def\problem{\@ifnextchar[{\@problemopt}{\@problemplain}}
\def\@problemplain{\paragraph{Problem} \it}
\def\@problemopt[#1]{\paragraph{Problem \normalfont #1}  \it}

\newenvironment{conjecture}[1][]
{\paragraph{Conjecture} {\normalfont #1} \it}
{\vspace{\the\theorempostskipamount}}
\def\conjecture{\@ifnextchar[{\@conjectureopt}{\@conjectureplain}}
\def\@conjectureplain{\paragraph{Conjecture} \it}
\def\@conjectureopt[#1]{\paragraph{Conjecture \normalfont #1}  \it}

\newenvironment{remarks}[1][]
{\paragraph{Remarks} {\normalfont #1}}
{\vspace{\the\theorempostskipamount}}
\def\remarks{\@ifnextchar[{\@remarksopt}{\@remarksplain}}
\def\@remarksplain{\paragraph{Remarks} \it}
\def\@remarksopt[#1]{\paragraph{Remarks \normalfont #1}  \it}

\newenvironment{example}[1][]
{\paragraph{Example} {\normalfont #1}}
{\vspace{\the\theorempostskipamount}}
\def\example{\@ifnextchar[{\@exampleopt}{\@exampleplain}}
\def\@exampleplain{\paragraph{Example}}
\def\@exampleopt[#1]{\paragraph{Example {\normalfont #1}}}

\newenvironment{examples}[1][]
{\paragraph{Examples} {\normalfont #1}}
{\vspace{\the\theorempostskipamount}}
\def\examples{\@ifnextchar[{\@examplesopt}{\@examplesplain}}
\def\@examplesplain{\paragraph{Examples} \it}
\def\@examplesopt[#1]{\paragraph{Examples \normalfont #1}  \it}

\newenvironment{exercise}[1][]
{\paragraph{Exercise} {\normalfont #1}}
{\vspace{\the\theorempostskipamount}}
\def\exercise{\@ifnextchar[{\@exerciseopt}{\@exerciseplain}}
\def\@exerciseplain{\paragraph{Exercise} \it}
\def\@exerciseopt[#1]{\paragraph{Exercise \normalfont #1}  \it}

\newenvironment{notation}[1][]
{\paragraph{Notation} {\normalfont #1}}
{\vspace{\the\theorempostskipamount}}
\def\notation{\@ifnextchar[{\@notationopt}{\@notationplain}}
\def\@notationplain{\paragraph{Notation} \it}
\def\@notationopt[#1]{\paragraph{Notation \normalfont #1}  \it}

\newenvironment{convention}[1][]
{\paragraph{Convention} {\normalfont #1}}
{\vspace{\the\theorempostskipamount}}
\def\convention{\@ifnextchar[{\@conventionopt}{\@conventionplain}}
\def\@conventionplain{\paragraph{Convention} \it}
\def\@conventionopt[#1]{\paragraph{Convention \normalfont #1}  \it}

\newenvironment{warning}[1][]
{\paragraph{Warning} {\normalfont #1}}
{\vspace{\the\theorempostskipamount}}
\def\warning{\@ifnextchar[{\@warningopt}{\@warningplain}}
\def\@warningplain{\paragraph{Warning} \it}
\def\@warningopt[#1]{\paragraph{Warning \normalfont #1}  \it}

\newenvironment{de'finition}[1][]
{\paragraph{Définition} {\normalfont #1} \it}
{\vspace{\the\theorempostskipamount}}

\newcommand{\thmendspace}{\vspace{\the\theorempostskipamount}}

\makeatother

\newenvironment{proof}[1][Proof]{\noindent \textit{#1.~}}
{\hfill $\Box$}

\newcommand{\dialogue}[1]%
{\textcolor{red}{$\ulcorner$ #1 
 $\lrcorner$}}

\newcommand{\aparte}[1]%
{$\ulcorner$ #1 
 $\lrcorner$}

\newlength{\parindentmem}
\makeatletter

\def\@removefromreset#1#2{\let\@tempb\@elt
   \def\@tempa#1{@&#1}\expandafter\let\csname @*#1*\endcsname\@tempa
   \def\@elt##1{\expandafter\ifx\csname @*##1*\endcsname\@tempa\else
         \noexpand\@elt{##1}\fi}%
   \expandafter\edef\csname cl@#2\endcsname{\csname cl@#2\endcsname}%
   \let\@elt\@tempb
   \expandafter\let\csname @*#1*\endcsname\@undefined}

\@addtoreset{paragraph}{section}
\@addtoreset{paragraph}{part}
\@removefromreset{paragraph}{subsubsection}
\@removefromreset{paragraph}{subsection}

\let\c@equation\c@subparagraph

\makeatother

\renewcommand{\theparagraph}{\arabic{paragraph}}
\renewcommand{\thesubparagraph}
{(\arabic{section}.\arabic{paragraph}.\arabic{subparagraph})}

\titleformat{\part}[display]{\normalfont\Large\bfseries}%
{\partname}{0cm}{}

\titleformat{\section}[hang]{\normalfont\Large\bfseries}{}{0cm}%
{\thesection \  --\ }

\titleformat{\subsection}[hang]{\normalfont\large\bfseries}{}{0cm}%
{\thesubsection\ -- \,}

\titleformat{\subsubsection}[hang]{\normalfont\bfseries}{}{0cm}{}

\newcommand{\spcifnec}[1]
{\ifx#1\empty
\else ~#1\!
\fi}

\titleformat{\paragraph}[runin]{\normalfont\bfseries}
{}{0cm}{}
\titlespacing{\paragraph}{0cm}
{1.7ex plus 1ex minus .2ex}
{.5em}

\titleformat{\subparagraph}[runin]{\it}
{\thesubparagraph}{0cm}{\spcifnec}
\titlespacing{\subparagraph}{0cm}
{0mm}
{.5em}

{\titleformat{\section}[hang]{\normalfont\large\bfseries}{}{0cm}{}
\titleformat{\subsection}[hang]{\normalfont\bfseries}{}{0cm}{}
\renewcommand{\theparagraph}{(\Alph{paragraph})}

\maketitle
}{}

\newenvironment{closing}%
{\titleformat{\section}[hang]{\normalfont\large\bfseries}{}{0cm}{}
\setlength{\itemsep}{0mm}
\small
}
{}

\makeatletter
\renewcommand\@maketitle{%
  \newpage
  \begin{center}%
  \let \footnote \thanks
    {\Large \bf \@title \par}%
    \vskip 1em%
    {\large
      \begin{tabular}[t]{c}%
        \@author
      \end{tabular}\par}%
  \end{center}%
  \par
  \vskip 1em}
\makeatother

\renewenvironment{abstract}
{\small \quotation
\noindent {\bf Abstract.}}{\endquotation \vskip .7cm}

\renewcommand{\P}{\mathbf{P}}
\newcommand{\PH}{\mathbf{P}\kern -.05em \mathrm{H}}

\renewcommand{\O}{\mathcal{O}}

\newcommand{\dlbrack}{[ \kern -.4ex [}
\newcommand{\drbrack}{] \kern -.4ex ]}

\renewcommand{\epsilon}{\varepsilon}
\renewcommand{\geq}{\geqslant}
\renewcommand{\leq}{\leqslant}

\newcommand{\taylordisc}{\smash
{(\raisebox{-.8mm}{\textpmhg{a}}\kern -.7mm)}}

\let\oldflat\flat
\def\flat#1{\vphantom{#1}^\oldflat \kern -.1em #1}

\begin{document}

\renewcommand{\O}{\mathcal{O}}

\setdefaultenum{(i)}{}{}{}

\title{Comment on:
On the irreducibility of the Severi variety of nodal curves in a
smooth surface, by E.~Ballico}
\author{Thomas Dedieu}

\maketitle

\begin{abstract}
In this short note, I point out that results of Ballico and 
Kool--Shende--Thomas together imply that on
$K3$, Enriques, and Abelian surfaces,
if $L$ is a very ample and $(2p_a(L)-2g-1)$-spanned line bundle,
then the equigeneric Severi variety 
$V_{g}(L)$ of all curves in $|L|$ having genus $g$
is non-empty, irreducible, of the expected dimension, and its general
member is a $(p_a(L)-g)$-nodal curve.
\end{abstract}

Let $S$ be a smooth, complex, projective surface, and $L$ an effective
line bundle on $S$. We denote by $p(L)$ the common arithmetic genus of
all members of the linear system $|L|$. For nonnegative integers $g$
and $\delta$, we consider the
\emph{equigeneric Severi variety} $V_g(L)$ 
(resp.\ \emph{nodal Severi variety} $V^\delta(L)$),
namely the locally closed subset in $|L|$ corresponding to reduced
curves of geometric genus $g$ (resp.\ with $\delta$ ordinary double
points and no further singularity).
In particular, $V^\delta(L)$ is an open subset of
$V_{p(L)-\delta}(L)$. 

In the recent paper \cite{ballico} Ballico has proven that if $L$ is
very ample and $(2\delta-1)$-spanned, then the nodal Severi variety
$V^\delta(L)$, if non-empty, is irreducible of codimension $\delta$
in $|L|$.
Here I show how this result can be enhanced by taking in consideration
a former result due to Kool, Shende and Thomas. 
This text is merely intended as a complement to \cite{ballico}, and I
thank Edoardo Ballico for giving me the opportunity to write this up.

\begin{theorem}
\label{t:Ktriv}
Let $S$ be a $K3$ (resp.\ Enriques, resp.\ Abelian) surface,
and $L$ a line bundle on it. 
Consider an integer $g \leq p(L)$. If $L$ is very ample and
$(2p(L)-2g-1)$-spanned, then the equigeneric Severi variety
$V_g(L)$ is non-empty and irreducible of dimension $g$ (resp.\
$g-1$, resp.\ $g-2$), and
the general member of $V_g(L)$ is a nodal curve.
\end{theorem}

\medskip
On $K3$, Enriques, and Abelian surfaces, there are explicit necessary
and sufficient conditions for a line bundle to be $k$-spanned, resp.\
$k$-very ample, 
\cite{BFS,BS,terakawa,andreas-kth}. In particular, they say that being
$k$-spanned and $k$-very ample are two equivalent conditions.

\begin{remark}
The arguments given here don't ensure that the general member
of $V_g(L)$ is irreducible. In practice, this may be obtained by
studying the various possible splittings of $L$ and a dimension
argument. 
\end{remark}

\medskip
It is now common knowledge that if $(S,L)$ is a polarized $K3$ or
Abelian surface, then the equigeneric Severi variety $V_g(L)$ is pure
of the expected dimension, see \cite{DS} and the references therein
(this is stated here in Proposition~\ref{p:estim}).
For a general such surface, it is also known that the nodal Severi
variety is non-empty by \cite{chen-rat} for $K3$'s 
and \cite{KL-gen2} for Abelian surfaces.
The density of the nodal Severi variety in the equigeneric one 
was so far only known if in addition $L$ is primitive (and $g \geq 5$
in the Abelian case),  
see \cite{chen-simple,chen-nodal}
(as well as \cite{DS,KLM}) for the $K3$ case, 
and \cite{KLM} for the Abelian case.

\begin{remark}
On Enriques surfaces, it is proved in \cite{bdgk-indam} that the
irreducible components of the nodal
Severi variety $V^{p(L)-g}(L)$ have dimension either $g-1$ or $g$.
In the range of application of Theorem~\ref{t:Ktriv}, there is only one
component of dimension $g-1$, and the condition given in
\cite{bdgk-indam} to distinguish between the two cases tells us that
for a general $[C] \in V^{p(L)-g}(L)$, the pull-back of $K_S$ to the
normalization of $C$ is non-trivial.
\end{remark}

\bigskip
As the main step in his proof, Ballico establishes the following
statement. 

\begin{proposition}
\label{p:ballico}
Let $L$ be a line bundle on a smooth complex projective surface.
If $L$ is very ample and $(2\delta-1)$-spanned,
then the family $\Sigma_\delta (L)$ 
of all members of $|L|$ which are singular in
(at least) $\delta$ points is irreducible of codimension $\delta$ in
$|L|$. 
\end{proposition}

\medskip
The result of Kool, Shende and Thomas that we use is the following,
see \cite[Prop.~2.1]{kst}.

\begin{proposition}
\label{p:kst}
Let $L$ be a line bundle on a smooth complex projective surface.
If $L$ is $\delta$-very ample, then the general $\delta$-dimensional
linear subsystem $\P^\delta \subseteq |L|$ contains a finite number of
$\delta$-nodal curves, 
and all other members are reduced curves of
geometric genus $p_g > p(L) -\delta$.
\end{proposition}

\medskip
This has the following immediate corollary: in the setting of the
proposition, if $V$ is an irreducible
variety of codimension $\leq \delta$ in $|L|$ parametrizing curves of
geometric genus $p_g \leq p(L)-\delta$, then the general member of $V$
is in fact a $\delta$-nodal curve.

\begin{corollary}
If $L$ is $\delta$-very ample and $(2\delta-1)$-spanned, then the
Severi variety of nodal curves $V^\delta(L)$ is non-empty and
irreducible of codimension $\delta$ in $|L|$.
\end{corollary}

\medskip
\begin{proof}
Every irreducible component of $V^\delta(L)$ is contained in
$\Sigma_\delta(L)$. On the other hand, $\Sigma_\delta(L)$ is
irreducible of codimension 
$\delta$ by Prop.~\ref{p:ballico}, and has an open subset contained in
$V^\delta(L)$ by Prop.~\ref{p:kst}.
\end{proof}

\bigskip
In the cases of Theorem~\ref{t:Ktriv}, one has the following estimates
on the dimensions of the Severi varieties.

\begin{proposition}
\label{p:estim}
Let $S$ be a $K3$ (resp.\ Enriques, resp.\ Abelian) surface,
$L$ an effective line bundle on $S$,
and $g \leq p(L)$ an integer.
Every irreducible component of the equigeneric Severi variety
$V_g(L)$ has dimension $=g$ (resp.\ $\geq g-1$, resp.\ $=g-2$).
\end{proposition}

\medskip
For Enriques surfaces, the estimate follows from
\cite[Lem.~2.3 and ineq.~(2.6)]{DS}. For $K3$ and Abelian surfaces a
well known extra argument is needed, see 
\cite[Prop.~4.5 and 4.13]{DS}.

\bigskip
\begin{proof}[Proof of Theorem~\ref{t:Ktriv}]%
As we have observed above, under the assumptions of
Theorem~\ref{t:Ktriv}, $L$ is actually $(2p(L)-2g-1)$-very ample,
hence also $(p(L)-g)$-very ample, so that both Propositions
\ref{p:ballico} and \ref{p:kst} apply for $\delta=p(L)-g$.
It follows that $V^\delta(L)$ is an irreducible, dense, non-empty, open
subset of $\Sigma_\delta(L)$.

On the other hand, let $V$ be an irreducible component of $V_g(L)$. 
By Prop.~\ref{p:estim}, $V$ has codimension $\leq \delta$ in
$|L|$. 
It thus follows from Prop.~\ref{p:kst} that the
general member of $V$ is a $\delta$-nodal curve, hence $V$ is
contained, and actually dense in $\Sigma_\delta(L)$.
This concludes the proof.
\end{proof}

\begin{closing}

\providecommand{\bysame}{\leavevmode\hbox to3em{\hrulefill}\thinspace}
\providecommand{\og}{``}
\providecommand{\fg}{''}
\providecommand{\smfandname}{and}
\providecommand{\smfedsname}{\'eds.}
\providecommand{\smfedname}{\'ed.}
\providecommand{\smfmastersthesisname}{M\'emoire}
\providecommand{\smfphdthesisname}{Th\`ese}

\medskip\noindent
Thomas Dedieu.
Institut de Mathématiques de Toulouse~; UMR5219.
Université de Toulouse~; CNRS.
UPS IMT, F-31062 Toulouse Cedex 9, France.
\texttt{thomas.dedieu@math.univ-toulouse.fr}

\end{closing}

\end{document}